\documentclass[10pt,english]{smfart}
\usepackage[all]{xypic}
\usepackage[T1]{fontenc}
\usepackage[english,french]{babel}
\usepackage{latexsym,amscd,color}
\usepackage{amsmath,amsfonts,amssymb,mathrsfs}
\usepackage{enumerate,euscript}
\usepackage{amssymb,url,xspace,smfthm}
\input xy
\xyoption{all}

\renewcommand{\a}{\mathfrak a}

\newcommand{\BibTeX}{{\scshape Bib}\kern-.08em\TeX}

\newcommand{\T}{\S\kern .15em\relax }
\newcommand{\AMS}{$\mathcal{A}$\kern-.1667em\lower.5ex\hbox
        {$\mathcal{M}$}\kern-.125em$\mathcal{S}$}
\newcommand{\resp}{\textit{resp}.\xspace}
\newcommand{\intersect}{\cdot\ldots\cdot}

\DeclareMathOperator{\proj}{Proj}

\DeclareMathOperator{\spec}{Spec}

\renewcommand{\P}{\mathbb{P}}

\newcommand{\Q}{\mathbb{Q}}

\renewcommand{\O}{\mathcal{O}}
\renewcommand{\L}{\overline{L}}
\newcommand{\sm}{\mathfrak{m}}

\newcommand{\f}{\mathbb{F}}
\newcommand{\ndot}{\raisebox{.4ex}{.}}

\allowdisplaybreaks

\numberwithin{equation}{section}

\title{Counting multiplicities in a hypersurface over number fields}
\date{\today}
\author{Hao Wen}
\address{Department of Mathematical Sciences\\Tsinghua University\\
Beijing 100084\\P. R. China}
\email{wen-h10@mails.tsinghua.edu.cn}
\author{Chunhui Liu}
\address{Department of Mathematics\\Faculty of Science\\Kyoto University\\Kyoto 606-8502\\Japan}
\email{chunhui.liu@math.kyoto-u.ac.jp}

\begin{document}
\def\smfbyname{}
\begin{abstract}
We fix a counting function of multiplicities of algebraic points in a projective hypersurface over a number field, and take the sum over all algebraic points of bounded height and fixed degree. An upper bound for the sum with respect to this counting function will be given in terms of the degree of the hypersurface, the dimension of the singular locus, the upper bounds of height, and the degree of the field of definition.
\end{abstract}

\maketitle

\tableofcontents

\section{Introduction}
In this paper, we consider a problem of counting multiplicities in projective schemes. More precisely, let $k$ be a field, and $X$ be a scheme of finite type over $\spec k$, we are interested in the estimate of the sum
\[\sum_{\xi\in S\left(X\left(\overline k\right)\right)}f\left(\mu_\xi\left(X\right)\right),\]
where $S\left(X\left(\overline k\right)\right)$ is a subset of $X(\overline k)$ which satisfies some conditions, and $f(\ndot)$ is a positive function and $\mu_{\xi}(X)$ is the multiplicity of $\xi$ in $X$ defined via the local Hilbert-Samuel function of $X$ at $\xi$ as follows.

We say that $X$ is a \textit{pure dimensional} scheme (or $X$ is of \textit{pure dimension}) if all its irreducible components have the same dimension. Let $X$ be of pure dimension, and $\xi\in X$ be a point. Consider the local ring $\O_{X,\xi}$, whose maximal ideal is $\sm_\xi$ and residue field is $\kappa(\xi)$. The \textit{local Hilbert-Samuel function} of $X$ at $\xi$ is given as
\[H_\xi(m)=\dim_{\kappa(\xi)}\left(\sm_\xi^m/\sm_\xi^{m+1}\right)\]
defined for all $m\in \mathbb{N}^+$. Suppose $\dim(\O_{X,\xi})=t\geqslant1$, then there exists a polynomial $P_\xi(T)$ of degree $t-1$ such that $H_\xi(m)=P_\xi(m)$ when $m$ is large enough. In addition, there exists an integer $\mu_\xi(X)\geqslant1$ such that
\begin{equation}\label{local Hilbert-Samuel}
P_\xi(m)=\mu_\xi(X)\frac{m^{t-1}}{(t-1)!}+o(m^{t-1}).
\end{equation}
We define the integer $\mu_{\xi}(X)$ as the \textit{multiplicity} of the point $\xi$ in $X$. In particular, if the point $\xi$ is regular in $X$, which means that $\O_{X,\xi}$ is a regular local ring, then we have $\mu_\xi(X)=1$.

If we take the counting function $f$ as the constant function $f\equiv1$, then this problem reduces to the classical problem of counting algebraic points on the scheme $X$. There have been many literatures about this problem hitherto. If we take $f$ to be a non-trivial function, and in addition we require $f(1)=0$, then this problem will be a question about the complexity of the singular locus of $X$.

\subsection{Known results}
First we consider the case where $X$ is a reduced plane curve of degree $\delta$. In Exercise 5-22 in page 115 of \cite{Fulton2}, we have
\begin{equation} \label{intro-fulton-multiplicity}
\sum\limits_{\xi\in X} \mu_{\xi}(X) \left(\mu_{\xi}(X)-1\right) \leqslant \delta(\delta-1),
\end{equation}
which is obtained by the B\'ezout's Theorem in the intersection theory. In addition, let $g(X)$ be the genus of $X$, if $X$ is geometrically integral, by Corollary 1 in page 201 of \cite{Fulton2}, we have
\[g(X) \leqslant \dfrac{(\delta-1)(\delta-2)}{2} - \sum_{\xi\in X}\dfrac{\mu_{\xi}(X)\left(\mu_{\xi}(X)-1\right)}{2}.\]
This inequality is deduced from the Riemann-Roch Theorem of plane curves.

More generally, let $X \hookrightarrow \mathbb{P}^n_k$ be a projective hypersurface over an algebraically closed field $k$ of characteristic $0$, whose singular locus is of dimension $0$. Through the method of Lefschetz pencil, a direct corollary of \cite[Corollaire 4.2.1]{Laumon1975} gives the inequality
\[\sum_{\xi\in X} \mu_\xi(X)(\mu_{\xi}(X)-1)^{n-1} \leqslant \delta(\delta-1)^{n-1}.\]
But the condition that the singular locus is of dimension $0$ is too restrictive for a general counting problem. In general, the sum in the left hand side of the above inequality depends on the candidate of the base field $k$.

In \cite[Th\'eor\`eme 5.1]{Liu-multiplicity}, the second author of this paper obtained a result of this type over finite fields. More precisely, let $n\geqslant2$, $\delta\geqslant1$, $s\geqslant0$ be three integers, and $\f_q$ be the finite field with exact $q$ elements. He proved that the estimate
\begin{equation} \label{Liu equality}
\sum_{\xi\in X(\f_q)} \mu_{\xi}(X)(\mu_\xi(X)-1)^{n-s-1}\ll_n\delta^{n-s}\max\{\delta-1,q\}^{s}
\end{equation}
holds uniformly for all reduced hypersurface $X$ of degree $\delta$ of $\mathbb P_K^n$ whose singular locus is of dimension $s$. In the above formula we have used  Vinogradov's symbol $\ll$ in its usual sense: let $\Omega$ and $P$ be two sets, and $\widetilde{\Omega}$ be a subset of $\Omega\times P$. Suppose that $f(x,y)$ and $g(x,y)$ are two real-valued functions defined on $\widetilde{\Omega}$, where $x\in\Omega$ and $y\in P$. Then the expression
\[f(x,y)\ll_{y}g(x,y)\]
means that there exists a non-negative function $C(\ndot)$ on the set $P$ such that
\[|f(x,y)|\leqslant C(y)|g(x,y)|\]
holds for every $(x,y)\in\widetilde{\Omega}$.

Some examples are given in \cite{Liu-multiplicity} to show that the order of $\delta$ and $\max\{\delta-1,q\}$ in \eqref{Liu equality} are both optimal when $q\geqslant\delta-1$. This estimate is obtained by the technique of intersection tree introduced in \cite[\S 2.1]{Liu-multiplicity} via the intersection theory on projective spaces.

\subsection{Principal Result}
In this paper, we consider a sum of the same type as in \eqref{Liu equality} over number fields. More precisely, we take the sum over all the algebraic points, whose fields of definition are of fixed degree over the base field, of bounded height in a hypersurface of a projective space. By the Northcott's property (cf. \cite[Theorem B.2.3]{Hindry}), this is a finite set, hence the sum always makes sense. The principle result (Theorem \ref{main theorem}) is stated as follows:
\begin{theo} \label{intro-main result}
Let $K$ be a number field, $n\geqslant2$ be an integer, and $h(\ndot)$ be the absolute logarithmic height function on $\mathbb P^n_K$. For any closed subscheme $X$ of $\mathbb P_K^n$, any $D\in\mathbb N^+$, and any $B\geqslant1$, let
\[S(X;D,B)=\{\xi\in X(\overline K)|\;[K(\xi):K]=D,\;\exp\left([K(\xi):\Q]h(\xi)\right)\leqslant B\}.\]
Let $\delta$ and $s$ be integers such that $\delta\geqslant 1$ and $s\geqslant 0$. Then the inequality
\begin{eqnarray*}& &\sum_{\xi\in S(X;D,B)} \mu_{\xi}(X)(\mu_{\xi}(X)-1)^{n-s-1}\\
&\leqslant&\sum_{t=0}^s\max_{Z\in\mathcal Z_t}\left\{\frac{\#S(Z;D,B)}{\deg(Z)}\right\}\delta(\delta-1)^{n-s+t-1},\end{eqnarray*}
holds for all reduce hypersurfaces $X$ of degree $\delta$ of $\mathbb P_K^n$ whose singular locus is of dimension $s$, where, for $t\in\{0,\ldots,s\}$, $\mathcal Z_t$ is a set of closed subschemes of $X$ of dimension $s-t$, whose construction will be explained in \S\ref{construction of intersection trees}.
\end{theo}
We keep all the notation in Theorem \ref{intro-main result}. If we want to get an upper bound of the sum
\[\sum\limits_{\xi\in S(X;D,B)} \mu_{\xi}(X)(\mu_{\xi}(X)-1)^{n-s-1}\]
through Theorem \ref{intro-main result} for all $X$ satisfying the above conditions, it is important to understand the term
\[\max\limits_{Z\in\mathcal Z_t}\left\{\frac{\#S(Z;D,B)}{\deg(Z)}\right\},\]
 which originates from the classical problem of counting algebraic points, or of counting rational points for the case of $D=1$.

We have the following corollary of Theorem \ref{intro-main result} for the case of $K=\Q$.
\begin{coro}[Corollary \ref{main corollary}]\label{intro-main corollary}
With all the notation and conditions of Theorem \ref{intro-main result}. Suppose $K=\Q$, and let $S(X;B)=S(X;1,B)$ for simplicity. Then the estimate
\[\sum_{\xi\in S(X;B)} \mu_{\xi}(X)(\mu_{\xi}(X)-1)^{n-s-1}\ll_{n}\delta^{n-s}\max\{B,\delta-1\}^{s+1},\quad B\geqslant 1\]
holds uniformly for all reduce singular hypersurfaces $X$ of degree $\delta$ of $\mathbb P_{\Q}^n$ whose singular locus is of dimension $s$.
\end{coro}
Moreover, we can construct some examples (for instance, Example \ref{cylinder}) to show that for all $X$ considered in Theorem \ref{intro-main result}, the exponents of $\delta$ and $\max\{B,\delta-1\}$ in Corollary \ref{intro-main corollary} are both optimal when $B\geqslant\delta-1$. We will also explain (Remark \ref{choice of counting function}) that the consideration in Theorem \ref{intro-main result} is necessary.

\subsection{Principal Tools}
We shall follow the construction of intersection trees introduced in \cite[\S 2.1]{Liu-multiplicity} to control the multiplicities of singular points. We construct a series of intersections over $\mathbb{P}^n_K$, and cut $X$ into several irreducible components. The multiplicity of each irreducible component can be bounded by its multiplicity in the intersection trees. Different from techniques used in \cite{Liu-multiplicity} over a finite field, we work over a number field in this paper, whose cardinality is infinite. Consequently this allows us to work over the original base field directly, and we do not need to take a finite extension of the base field in order to make sure that we can construct useful auxiliary schemes, and then descend it back to the original base field.

Meanwhile, we need to consider the number of rational points and algebraic points of bounded height. Since we require that the constant in the estimate in Corollary \ref{intro-main corollary} only depends on $n$, we need a uniform estimate of the number of algebraic points of bounded height in arithmetic varieties, which has a weak dependance on the degrees of varieties.

This paper is organized as follows: in \S 2, we introduce the technique of intersection tree in \cite{Liu-multiplicity}. In \S 3, we recall some useful results on counting rational points and algebraic points, and we consider a uniform estimate of rational points of bounded height over $\Q$, which is a generalization of \cite[Theorem 1]{Schanuel} and \cite[Theorem 3.1]{Browning-PM277}. In \S 4, we give an upper bound of this multiplicity-counting problem as a function of intersection trees, and we give a uniform upper bound of it via a generalized Schanuel's estimate.

\subsection*{Acknowledgment}
We would like to thank Dr. Yang Cao and Dr. Enlin Yang for some useful suggestions on some technical details in this paper, and we would like to thank the anonymous referees for their useful comments and suggestions. Chunhui Liu is supported by JSPS KAKENHI Grant Number JP17F17730.

\section{Operations over intersection trees}
In this section, we recall the notion of intersection tree in the settings of graph theory and some useful properties of it. These are introduced in \cite{Liu-multiplicity}. We fix a base field $k$ throughout this section.

\subsection{Preliminaries of intersection theory} \label{subsection_intersection_theory}
Let $X$ be a projective scheme and $\xi\in X$. In \eqref{local Hilbert-Samuel}, we have defined the multiplicity of the point $\xi$ in $X$, noted by $\mu_\xi(X)$. In addition, if $M$ is an integral closed subscheme of $X$ whose generic point is $\xi_M$, we define the \textit{multiplicity} of $M$ in $X$ as $\mu_{\xi_M}(X)$, noted by $\mu_M(X)$ for simplicity.

In the following, we will recall some useful notions and properties of the intersection theory. We will follow the strategy of \cite{SerreLocAlg} and \cite{Fulton}.

Let $Y$ be a separated regular $k$-scheme of finite type, $r \geqslant 2$ be an integer, and $X_1,\ldots, X_r$ be pure dimensional closed subschemes of $Y$. We denote by $\mathcal{C}(X_1\intersect X_r)$ the set of irreducible components of the intersection product $X_1\intersect X_r$. Let $X$ be a pure dimensional closed subscheme of $Y$, we denote by $\mathcal{C}(X)$ the set of irreducible components of $X$. If not specially mentioned, each element of $\mathcal{C}(X_1\intersect X_r)$ and $\mathcal C(X)$ is considered to be an integral closed subscheme of $Y$. Let $M\in\mathcal C(X_1\intersect X_r)$, we denote by
\[i(M;X_1\intersect X_r;Y)\]
the \textit{intersection multiplicity} of the intersection product $X_1\intersect X_r$ at $M$, and we refer readers to \cite[Chapter 7 and 8]{Fulton} for its definition.

Let $M\in\mathcal{C}(X_1\intersect X_r)$, with $r\geqslant2$. In general, we have (cf. \cite[Chap. III, Prop. 17]{SerreLocAlg})
\[\dim(M)\geqslant\dim(X_1)+\cdots+\dim(X_r)-(r-1)\dim (Y).\]
If the equality holds and the intersection is not empty, we say that $X_1,\ldots,X_r$ \textit{intersect properly} at $M$ in $Y$, and $M$ is a \textit{proper component} of the intersection product $X_1\intersect X_r$ in $Y$. If $X_1,\ldots,X_r$ intersect properly at all its irreducible components, we say that $X_1,\ldots,X_r$ \textit{intersect properly}.

\subsubsection*{B\'ezout's Theorem}
Let $Y$ be a regular projective $k$-scheme and $\mathscr{L}$ be an ample invertible $\O_Y$-module. If $X$ is a closed subscheme of $Y$, we denote by $\deg_{\mathscr{L}}(X)$ the degree of $X$ with respect to the invertible $\O_Y$-module $\mathscr{L}$, which is defined as $\deg\left(c_1(\mathscr{L})^{\dim(X)}\cap[X]\right)$. If $\mathscr{L}$ is the universal bundle $\O_Y(1)$, we note the degree by $\deg(X)$ for simplicity.

The B\'ezout's Theorem is a description of the complexity of a proper intersection in $\mathbb P^n_k$ in terms of degrees with respect to the universal bundles.
\begin{theo}[B\'ezout's Theorem]\label{bezout}
Let $X_1,\ldots,X_r$ be a family of closed pure dimensional subschemes of $\mathbb P^n_k$, which intersect properly. Then we have
\begin{equation*}
\sum_{Z\in\mathcal C(X_1\intersect X_r)}i(Z;X_1\intersect X_r;\mathbb P^n_k)\deg(Z)=\deg(X_1)\cdots\deg(X_r).
\end{equation*}
\end{theo}
We refer readers to \cite[Proposition 8.4]{Fulton} for more details. See also the equality (1) in page 145 of \cite{Fulton}.

\subsection{Definition of intersection tree}
Let $Y$ be a regular separated $k$-scheme and $\mathscr{L}$ be an ample invertible $\O_Y$-module. Let $\delta\geqslant1$ be an integer. We call a directed rooted tree $\mathscr{T}$ with labelled vertices and weighted edges an \textit{intersection tree of level $\delta$} over $Y$, if it satisfies the following conditions:
\begin{enumerate}
    \item the vertices of $\mathscr T$ are the occurrences of integral closed subschemes of $Y$ (an integral closed subscheme of $Y$ can appear several times in a tree);
    \item each vertex $X$ of $\mathscr T$ is attached with a label, which is a pure dimensional closed subscheme of $Y$ or empty;
    \item a vertex of $\mathscr T$ is a leaf if and only if its label is empty;
    \item if $X$ is a vertex of $\mathscr T$ which is not a leaf, then
    \begin{itemize}
      \item its label $\widetilde{X}$ satisfies the inequality $\deg_{\mathscr{L}}(\widetilde{X})\leqslant \delta$ and the closed subschemes $X$ and  $\widetilde{X}$ intersect properly in $Y$;
      \item the children of $X$ are precisely the irreducible components of the intersection product $X\cdot \widetilde{X}$ in $Y$;
      \item for each child $Z$ of $X$, the edge $\ell$ which links $X$ and $Z$ is attached with a weight $w(\ell)$ which equals the intersection multiplicity $i(Z;X\cdot \widetilde{X}; Y)$.
    \end{itemize}
\end{enumerate}

For every fixed intersection tree $\mathscr T$, we call any of the complete subtrees of $\mathscr T$ an \textit{sub-intersection tree}, which is necessarily an intersection tree.
\subsubsection*{Weight of a vertex}
Let $Y$ be a regular separated scheme over $\spec k$, equipped with an ample invertible sheaf $\mathscr{L}$, and $\mathscr T$ be an intersection tree over $Y$. For each vertex $X$ of $\mathscr T$, we define the \textit{weight} of $X$ as the product of the weights of all edges in the path which links the root of $\mathscr T$ and the vertex $X$, denoted as $w_{\mathscr T}(X)$. If $X$ is the root of an intersection tree, we define $w_{\mathscr T}(X)=1$ for convenience.

\subsubsection*{Weight of an integral closed subscheme}
Let $Z$ be an integral closed subscheme of $Y$. We define the \textit{weight} of $Z$ relative to the tree $\mathscr{T}$ as the sum of the weights of all the occurrences of $Z$ as vertices of $\mathscr{T}$, noted by $W_{\mathscr{T}}(Z)$. If $Z$ does not appear in the tree $\mathscr T$ as a vertex, for convenience the weight $W_{\mathscr{T}}(Z)$ is defined to be $0$. Let $Z$ be a vertex in the intersection tree $\mathscr{T}$. When we write $W_{\mathscr{T}}(Z)$, the symbol $Z$ is considered as an integral closed subscheme of $Y$. In other words, we count all the occurrences of the subscheme $Z$ in the intersection tree $\mathscr{T}$.

\subsubsection*{Example of intersection trees}
We refer the readers to \cite[Exemple 3.2]{Liu-multiplicity} as an example of the notion of intersection tree.
\subsection{Estimate of weights of intersection trees}
In order to estimate the weights in intersections trees, we first introduce the following result.
\begin{theo}[Th\'eor\`eme 3.1, \cite{Liu-multiplicity}]\label{mult in the intersection tree}
Suppose that $k$ is a perfect field. Let $\{X_i\}_{i=1}^r$ be a family of closed pure dimensional subschemes of $\mathbb P^n_k$ which intersect properly in $\mathbb P^n_k$. For each irreducible component $C \in \mathcal C(X_1\intersect X_r)$, let $\mathscr T_C$ be an intersection tree whose root is $C$. We consider a vertex $M$ in the intersection trees $\{\mathscr T_C\}_{C\in\mathcal C(X_1\intersect X_r)}$ which satisfies: for each vertex $Z$ in $\{\mathscr T_C\}_{C\in\mathcal C(X_1\intersect X_r)}$, if $M$ is a proper subscheme of $Z$, then there exists an occurrence of $M$ as a descendant of $Z$. Then we have
\begin{equation}\label{no auxillary scheme}
\sum_{C\in\mathcal C(X_1\intersect X_r)}W_{\mathscr T_C}(M)i(C;X_1\intersect X_r;\mathbb P^n_k)\geqslant \mu_{M}(X_1)\cdots\mu_{M}(X_r),
\end{equation}
where $\mu_M(X_i)$ is the multiplicity of $M$ in $X_i$, defined in \eqref{local Hilbert-Samuel}.
\end{theo}
Keeping all the notation in Theorem \ref{mult in the intersection tree}, we introduce the following notions.
\begin{defi} \label{def of C_s}
Let $s$ be a non-negative integer. We define $\mathcal C_s$ as the set of all vertices of depth $s$ in the intersection trees $\mathscr T_C$, where $C\in\mathcal C(X_1\intersect X_r)$. In addition, we define $\mathcal C_*=\bigcup\limits_{s\geqslant0}\mathcal C_s$.
\end{defi}
We define a subset of $\mathcal C_s$ for each non-negative integer $s$ as below.
\begin{defi} \label{def of Z_s}
Let $s$ be a non-negative integer. We define $\mathcal Z_s$ as the subset of $\mathcal C_s$ of elements $M$ which satisfy the following condition: for every vertex $Z$ of intersection trees $\{\mathscr T_C\}_{C\in\mathcal C(X_1\intersect X_r)}$, if $M$ is a proper subscheme of $Z$, then there exists a descendant of $Z$ which is an occurrence of $M$. In addition, we define $\mathcal Z_*=\bigcup\limits_{s\geqslant0}\mathcal Z_s$.
\end{defi}
By definition, we have $\mathcal Z_0=\mathcal C_0=\mathcal C(X_1\intersect X_r)$. In fact, Theorem \ref{mult in the intersection tree} is satisfied for every element in $\mathcal Z_*$. \begin{defi}
Let $s$ be a non-negative integer. We denote by $\mathcal C'_s$ (\resp $\mathcal Z'_s$, $\mathcal C'_*$ and $\mathcal Z'_*$) the set of the labels of $\mathcal C_s$ (\resp $\mathcal Z_s$, $\mathcal C_*$ and $\mathcal Z_*$).
\end{defi}
The following proposition is a corollary of Theorem \ref{mult in the intersection tree}, which is proved via Theorem \ref{bezout}.
\begin{prop}[Proposition 4.6, \cite{Liu-multiplicity}] \label{grassmanne}
With all the above notation and the conditions in Theorem \ref{mult in the intersection tree}, we suppose that all the non-empty elements in $\mathcal C'_*$ have the same dimension. Then we have
\begin{equation*}
\sum_{Z\in \mathcal Z_s} \left(\prod_{i=1}^r\mu_Z(X_i)\right)\deg(Z) \leqslant \prod_{i=1}^r\deg(X_i)\prod_{j=0}^{s-1}\max_{\widetilde{Z}\in \mathcal C'_j}\{\deg(\widetilde{Z})\},
\end{equation*}
where we define by convention $\prod\limits_{j=0}^{s-1}\max\limits_{\widetilde{Z}\in \mathcal C'_j}\{\deg(\widetilde{Z})\}=1$ if $s=0$.
\end{prop}
\section{Counting algebraic points in arithmetic varieties}

Let $K$ be a number field. In order to describe the arithmetic complexity of the closed points in $\mathbb P^n_K$, we introduce the following height function.
\subsection{Definition of height functions}
Let $K$ be a number field, $\overline K$ be an algebraic closure of $K$, and $M_K$ be the set of all places of $K$. For every element $x\in K$, we define the absolute value $|x|_v=\left|N_{K_v/\Q_v}(x)\right|_v^\frac{1}{[K_v:\Q_v]}$ for each $v\in M_K$, extending the usual absolute values on $\Q_p$ or $\mathbb{R}$. In addition, we define $\|\ndot\|_v=|\ndot|_v^{[K_v:\Q_v]}$ for every $v\in M_K$.
\begin{defi}\label{weil height}
Let $\xi\in\mathbb{P}^n_K(\overline K)$ be a closed point and $K'$ be any field such that $[K':K]<+\infty$ and $\xi\in\mathbb P^n_K(K')$. We write a $K'$-rational homogeneous coordinate of $\xi$ as $[x_0:\cdots:x_n]$. We define the \textit{absolute logarithmic height} of the point $\xi$ as
\[h(\xi)=\sum_{v\in M_{K'}}\frac{1}{[K':\Q]}\log\left(\max_{0\leqslant i\leqslant n}\{\|x_i\|_v\}\right),\]
which is independent of the choice of the projective coordinate by the product formula (cf. \cite[Chap. III, Proposition 1.3]{Neukirch}).
\end{defi}
We can prove that $h(\xi)$ is independent of the choice of the field $K'$ (cf. \cite[Lemma B.2.1]{Hindry}).

If $\xi$ is an algebraic point of $\mathbb P^n_K$ valued in a number field $K'$ containing $K$, we define the \textit{relative multiplicative height} of the point $\xi$ to be
\[H_{K'}(\xi) = \exp\left([K':\Q]h(\xi)\right).\]

When considering the closed points of a subscheme $X$ of $\mathbb P^n_K$ with the immersion $\phi:X\hookrightarrow\mathbb P^n_K$, we define the height of $\xi\in X(\overline K)$ to be
\[h(\xi):=h(\phi(\xi)).\]
We shall use this notation when there is no confusion of the immersion morphism $\phi$.

Let $B \geqslant1$, $D\in\mathbb N^+$, and $X$ be the subscheme of $\mathbb P^n_K$ defined above. We denote by
\begin{equation}\label{S(X;D,B)}
S(X;D,B)=\{\xi\in X(\overline K)|\;[K(\xi):K]=D, H_{K(\xi)}(\xi)\leqslant B\},
\end{equation}
where $K(\xi)$ is the residue field of $\xi$ in $\mathbb P^n_K$. In particular, we denote by
\begin{equation}\label{S(X;B)}
S(X;B)=S(X;1,B)=\{\xi\in X(K)|\;H_K(\xi)\leqslant B\}
\end{equation}
for simplicity. We denote
\begin{equation}\label{N(X;D,B)}
N(X;D,B)=\#S(X;D,B)
\end{equation}
and
\begin{equation}\label{N(X;B)}
N(X;B)=\#S(X;B).
\end{equation}
By the Northcott's property (cf. \cite[Theorem B.2.3]{Hindry}), $N(X;D,B)$ is finite for every $D\in \mathbb N^+$ and every $B\geqslant1$.

For the problem of counting rational points or algebraic points, it is essential to understand the functions $N(X;B)$ and $N(X;D,B)$ in variables $B$ and $D$. There are fruitful results on this topic, and we will introduce some which are useful in the multiplicity-counting problem.

\subsection{Schanuel's estimate}
Let $B \geqslant1$, $D\in\mathbb N^+$ and $X\hookrightarrow\mathbb P^n_K$ be a projective scheme. With all the notation above, it is natural to consider the density of algebraic points via some properties of $N(X;D,B)$ and $N(X;B)$. First we consider the case where $X=\mathbb P^n_K$.

\subsubsection{The density of rational points of projective spaces}
For $N(\mathbb P^n_K;B)$, we have the following asymptotic estimate
\begin{equation}\label{Schanuel estimate}
N(\mathbb P^n_K;B)=\alpha(K,n)B^{n+1}+o(B^{n+1}),\quad B\rightarrow+\infty
\end{equation}
for all $n\in \mathbb Z^+$, where the constant $\alpha(K,n)$ is articulated in the paper of S. Schanuel \cite[Theorem 1]{Schanuel}.

For the case of $K=\Q$. Let $\xi\in\mathbb P^n_{\Q}(\Q)$, we take the primitive projective coordinate of $\xi$ as $[\xi_0:\cdots:\xi_n]$, which means each $\xi_i\in\mathbb Z$ and $\gcd(\xi_0,\ldots,\xi_n)=1$. In this case, we have
\[H_{\Q}(\xi)=\max_{0\leqslant i\leqslant n}\{|\xi_i|\},\]
where $|\ndot|$ is the usual absolute value. In addition, we have
\begin{equation}\label{Schanuel over Q}
N(\mathbb P^n_{\Q};B)=\frac{2^n}{\zeta(n+1)}B^{n+1}+o(B^{n+1}),\quad B\rightarrow+\infty
\end{equation}
for all $n\in \mathbb N^+$, where $\zeta(n)$ is the usual Riemann zeta function. We refer to \cite[Theorem 1.2]{Browning-PM277} for a proof, which is simpler than that of \cite[Theorem 1]{Schanuel}. In this case, we have an explicit uniform estimate of $N(\mathbb P^n_{\Q};B)$ as following.
\begin{prop}
The inequality
  \[N(\mathbb P^n_{\Q};B)\leqslant 3^{n+1}B^{n+1}\]
holds for all $B\geqslant1$ and $n\in\mathbb N^+$.
\end{prop}
\begin{proof}
  We consider the set
  \[R(\mathbb A_{\mathbb Z}^{n+1};B)=\left\{\xi=(\xi_0,\ldots,\xi_n)\in\mathbb A_{\mathbb Z}^{n+1}(\mathbb Z)\mid\;\max_{0\leqslant i\leqslant n}\{|\xi|\}\leqslant B\right\}.\]
  Because there are at most $2B+1$ integers whose absolute values are smaller than $B$, we have
  \[\#R(\mathbb A_{\mathbb Z}^{n+1};B)\leqslant(2B+1)^{n+1}\leqslant3^{n+1}B^{n+1}.\]
  In addition, we have $N(\mathbb P^n_{\Q};B)\leqslant\#R(\mathbb A_{\mathbb Z}^{n+1};B)$. So we get the result.
\end{proof}

\subsubsection{The density of algebraic points of projective spaces}\label{counting algebraic points in Pn}
We have discussed $N(\mathbb P^n_K;B)=N(\mathbb P^n_K;1,B)$ above for the case of rational points. For $N(\mathbb P^n_K;D,B)$ with arbitrary $D\in \mathbb N^+$, the situation is very different. Until now, to the authors' knowledge, there is no optimal asymptotic estimate of $N(\mathbb P^n_K;D,B)$ for general $n$, $D$ and $K$. We only have some partial results for these $n$, $D$ and $K$ satisfying certain conditions.

Let $A(K,n,D)$ be a series of positive constants depending on $n,D\in\mathbb N^+$ and the number field $K$. First we consider the case of $n=1$, in which case $\mathbb P^n_K$ is a projective line. In this case, we have
\[N(\mathbb P^1_K;D,B)\sim A(K,1,D)B^{D+1}\]
for all $D\in\mathbb N^+$ and for all number field $K$, see \cite{MasserVaaler2007} or \cite[Th\'eor\`eme 5.1]{LeRudulier2014} for a proof, where the constant $A(K,1,D)$ is explicitly given in the above two references.

Higher dimensional cases are more complicated. Actually, when $n\geqslant3$, we have
\[N(\mathbb P^n_K;2,B)\sim A(K,n,2)B^{n+1}\]
for all $B\geqslant1$ and arbitrary number field $K$ in \cite[Theorem 1.2.1]{Guignard2017}), where the constant $A(K,n,2)$ is given explicitly (loc. cit.). The case of $K=\Q$ is treated in \cite{Schmidt1995}.

For the cases of higher extension degrees $D\in\mathbb N^+$, we have
\[N(\mathbb P^n_K;D,B)\sim A(K,n,D) B^{n+1}\]
for all $n,D\in \mathbb N^+$ satisfying when $n\geqslant D+2$, or a better estimate
\[N(\mathbb P^n_K;D,B)\ll_{K,D}B^{D+1+\frac{n-1}{D}}\log B,\quad B\rightarrow+\infty\]
 holds uniformly for all $D\geqslant n\geqslant3$, see \cite[Theorem 1.2.2]{Guignard2017} for the constant $A(K,n,D)$ involved above. For the case of $K=\Q$, this is a theorem in \cite{GaoXiaThesis}.

\subsection{A naive estimate for arithmetic varieties}

The aim of this section is to prove the following result.
\begin{theo}\label{refined Schanuel estimate for rational points}
Let $n\geqslant2$, $\delta\geqslant1$ and $d\geqslant1$ be three integers. Then the estimate
  \[S(X;B)\ll_{n}\delta B^{d+1},\quad B\geqslant 1\]
  holds uniformly for all pure dimensional closed subscheme $X$ of $\mathbb P^n_{\Q}$ of dimension $d$ and degree $\delta$.
\end{theo}
In order to prove it, we will introduce auxiliary results. First, we introduce the following definition.
\begin{defi}\label{definition of degree of affine scheme}
  Let $k$ be a field, and $X$ be a closed subscheme of $\mathbb A^n_k$. We define the \textit{degree} of $X$ in $\mathbb A^n_k$ to be the degree of its projective closure in $\mathbb P^n_k$. The degree of $X$ defined above is denoted by $\deg(X)$ if there is no confusion.
\end{defi}
By Definition \ref{definition of degree of affine scheme}, we have the following result.
\begin{lemm}\label{degree of affine variety}
  Let $k$ be a field, $X\hookrightarrow\mathbb A_k^n$ be a pure dimensional closed subscheme of dimension $d$ with $d\geqslant1$, and $L$ be a linear subscheme of $\mathbb A_k^n$ which intersects $X$ properly. Then we have
  \[\deg(X)\geqslant\sum_{Z\in\mathcal C(X\cap L)}\deg(Z),\]
where $\deg(\ndot)$ follows Definition \ref{definition of degree of affine scheme}, and $\mathcal C(X\cap L)$ is the set of irreducible components of $X\cap L$ considered to be integral closed subschemes of $\mathbb A^n_k$. Moreover, we define $\deg(Z)=1$ by convention if $Z$ is a closed point.
\end{lemm}
\begin{proof}
  Let $\overline X$ and $\overline L$ be the projective closure of $X$ and $L$ in $\mathbb P^n_k$ respectively, then we have $\deg(X)=\deg(\overline X)$ and $\deg(L)=\deg(\overline L)=1$ by Definition \ref{definition of degree of affine scheme}. By the B\'ezout's Theorem (Theorem \ref{bezout}), we have
  \[\deg(\overline X)=\deg(\overline X)\deg(\overline L)=\sum_{Z\in\mathcal C(\overline X\cdot\overline L)}i(Z;\overline X\cdot\overline L;\mathbb P^n_k)\deg(Z),\]
  where $\mathcal C(\overline X\cdot\overline L)$ is the set of irreducible components of the intersection $\overline X\cdot\overline L$, and $i(Z;\overline X\cdot\overline L;\mathbb P^n_k)$ is the intersection multiplicity of $\overline X\cdot \overline L$ at $Z$. For each $Z\in\mathcal C(\overline X\cdot\overline L)$, let $a(Z)$ be the restriction of $Z$ in $\mathbb A^n_k$ involved above. By Definition \ref{definition of degree of affine scheme}, we have $\deg(Z)=\deg(a(Z))$ if $a(Z)\neq\emptyset$. So we obtain
  \[\sum_{Z\in\mathcal C(\overline X\cdot\overline L)}i(Z;\overline X\cdot\overline L;\mathbb P_k^n)\deg(Z)\geqslant\sum_{Z\in\mathcal C(\overline X\cdot\overline L)}\deg(a(Z))=\sum_{Z\in\mathcal C(X\cap L)}\deg(Z),\]
  where we define $\deg(a(Z))=0$ if $a(Z)=\emptyset$ above. The reason is that each intersection multiplicity is larger than or equal to $1$. So we have the result.
\end{proof}
We need the following lemma about the definition of Krull dimension of a topological space. We refer the reader its definition at \cite[Definition 2.5.1]{LiuQing}.
\begin{lemm}\label{irreudicible closed subset}
  Let $k$ be a field, and $X$ be a non-empty closed irreducible subset of the affine space $\mathbb A^n_k$ whose dimension is $d$, where $d\geqslant0$. Then $X$ has no proper closed subset of dimension $d$. A proper subset of $X$ means a subset of $X$ which is not equal to $X$ itself.
\end{lemm}
\begin{proof}
  We suppose that $X$ has a proper irreducible closed subset $X'$ of dimension $d$. Let
  \[X'=X_0\supsetneq X_1\supsetneq\cdots\supsetneq X_d\]
  be a sequence of non-empty irreducible closed subsets of $X'$. Then we have the following sequence of non-empty irreducible closed subsets of $X$
  \[X\supsetneq X_0\supsetneq X_1\supsetneq\cdots\supsetneq X_d,\]
  which shows that the dimension of $X$ is at least $d+1$. This leads to a contradiction.
\end{proof}
Next, we prove a lemma about the intersection of affine schemes.
\begin{lemm}\label{intersection with a hyperplane}
  Let $k$ be a field, and $X$ be an irreducible closed subscheme of $\mathbb A^n_{k}=\spec\left(k[T_1,\ldots,T_n]\right)$, which is of dimension $d$ with $1\leqslant d\leqslant n-1$. Then there exists an index $\alpha\in\{1,\ldots,n\}$, such that for all $a\in k$, the hyperplane defined by the equation $T_\alpha=a$ intersects $X$ properly.
\end{lemm}
\begin{proof}
  For $\alpha\in\{1,\ldots,n\}$ and $a\in k$, we denote by $H(T_\alpha=a)$ the hyperplane defined by the equation $T_\alpha=a$. By \cite[Chap. III, Prop. 17]{SerreLocAlg}, we have $\dim(X\cap H(T_\alpha=a))\geqslant d+n-1-n=d-1$ for all $\alpha\in\{1,\ldots,n\}$ and all $a\in k$. By \cite[Proposition 2.5.5 (a)]{LiuQing}, we have
  \[\dim(X\cap H(T_\alpha=a))\leqslant\dim(X)=d.\]
   If for each $\alpha\in\{1,\ldots,n\}$, we can find an element $a\in k$ such that $X\cap H(T_\alpha=a)$ is not a proper intersection, which means that we have $\dim(X\cap H(T_\alpha=a))=d$ by definition directly.

  The set $X\cap H(T_\alpha=a)$ is a closed subset of $X$ and $H(T_\alpha=a)$ by the definition of topological space. By Lemma \ref{irreudicible closed subset}, there is no proper closed subset of $X$ whose dimension is $d$ since the scheme $X$ is irreducible and $\dim(X)=d$. From the fact $\dim(X\cap H(T_\alpha=a))=d$, we have $X=X\cap H(T_\alpha=a)$. So we obtain $X\subseteq H(T_\alpha=a)$.

   From the above hypothesis, for all $\alpha\in\{1,\ldots,n\}$, there exists an element $a\in k$ such that $X\subseteq H(T_\alpha=a)$. For every $\alpha\in\{1,\ldots,n\}$, we choose one of these elements in $k$, noted by $a_n$. Then we have $X\subseteq\;H(T_1=a_1)\cap\cdots\cap H(T_n=a_n)$. The scheme $H(T_1=a_1)\cap\cdots\cap H(T_n=a_n)$ is the rational point in $\mathbb A^n_{k}$ whose affine coordinate is $(a_1,\ldots,a_n)$, so we have
  \[X\subseteq (a_0,\ldots,a_n).\]
  This contradicts the fact that $d\geqslant1$. So we prove the result.
\end{proof}
Now we prove a proposition of counting integral points in affine schemes. Before doing this, we introduce a definition of $\mathbb Z$-points of a $\Q$-scheme.

Let $\phi:X\hookrightarrow\mathbb A^n_{\mathbb Q}$ be an arbitrary affine subscheme of $\mathbb A^n_{\Q}$, then we have the following diagram:
\[\xymatrix{X\ar@{^{(}->}^{\phi}[r]&\mathbb A^n_{\Q}\ar[r]^\pi\ar[d]\ar@{}|-{\square}[dr]&\mathbb A^n_{\mathbb Z}\ar[d]\\
&\spec \Q\ar[r]&\spec \mathbb Z.}\]
\begin{defi}\label{Z-point of a Q-scheme}
With the above construction, we denote by $X_{\phi}(\mathbb Z)$ the subset of $X(\Q)$ of the $\xi\in X(\Q)$ (considered as $\Q$-morphisms from $\spec\Q$ to $X$) whose composition with the canonical immersion morphism $\phi:X\hookrightarrow \mathbb A^n_\Q$ gives a $\mathbb Z$-point of $\mathbb A^n_{\mathbb Z}$ having the value in $\mathbb Q$ which comes from a $\mathbb Z$-point of $\mathbb A^n_{\mathbb Z}$. In other words, we define $X_{\phi}(\mathbb Z)=X(\mathbb Q)\cap\pi^{-1}(\mathbb A^n_{\mathbb Z}(\mathbb Z))$. Instead of $X_{\phi}(\mathbb Z)$, we denote this set by $X(\mathbb Z)$ if there is no confusion of the morphism $\phi$.
\end{defi}

\begin{prop}\label{affine cone lemma}
  For any $B\geqslant1$, and any subscheme $X$ of $\mathbb A^{n+1}_{\Q}$, let
  \[M(X;B)=\left\{\xi=(\xi_0,\ldots,\xi_n)\in X(\mathbb Z) \middle| \;\max_{0\leqslant i\leqslant n}\{|\xi_i|\}\leqslant B\right\},\]
where the set $X(\mathbb Z)$ is defined in Definition \ref{Z-point of a Q-scheme}. Let $\delta\geqslant1$, $d\geqslant1$ and $n\geqslant1$ be three integers, then the estimate
  \[\#M(X;B)\ll_{n}\delta B^{d},\quad B\geqslant 1\]
holds uniformly for all pure dimensional closed subschemes $X$ of $\mathbb A^{n+1}_{\Q}$ of dimension $d$ and degree $\delta$.
\end{prop}
\begin{proof}
We can suppose that $X$ is irreducible, else we can count it component by component.

We reason by induction on $d$ to prove this lemma. If $d=1$, by Lemma \ref{intersection with a hyperplane}, there exists an index $\alpha\in\{1,\ldots,n\}$ such that $X$ intersects the hyperplane defined by $T_\alpha=a$ properly. Let $H_a$ denote this hyperplane. Then we have
 \begin{equation}
   M(X;B)=\bigcup_{\begin{subarray}{c}a\in\mathbb Z\\|a|\leqslant B \end{subarray}}M(X\cap H_a;B), \quad B\geqslant1.
 \end{equation}
  From Lemma \ref{degree of affine variety}, each set $M(X\cap H_a;B)$ contains at most $\delta$ closed points. By the fact that there are at most $2B+1$ integers whose absolute values are smaller than $B$, we get
 \[M(X;B)\leqslant\delta(2B+1),\quad B\geqslant1,\]
 which proves the case of $d=1$.

 Next, we suppose that $d\geqslant2$. In this case, by Lemma \ref{intersection with a hyperplane}, we can find an index $\alpha\in\{0,\ldots,n\}$ such that $X$ intersects the hyperplane defined by $T_\alpha=a$ properly for any $a\in \mathbb Q$. Let $H_a$ denote this hyperplane. Then we have
 \begin{equation}\label{intersected by hyperplane}
   M(X;B)=\bigcup_{\begin{subarray}{c}a\in\mathbb Z\\|a|\leqslant B \end{subarray}}M(X\cap H_a;B),\quad B\geqslant1.
 \end{equation}

 For every $a\in\mathbb Z$ above, the scheme $X\cap H_a$ has dimension at most $d-1$. By Lemma \ref{degree of affine variety}, we have the inequality
\[\delta=\deg(X)\geqslant\sum_{Z\in \mathcal C(X\cap H_a)}\deg(Z),\]
where $\mathcal C(X\cap H_a)$ is the set of irreducible components of $X\cap H_a$.

By the induction hypothesis, we have
\[\#M(Z;B)\ll_{n}\deg(Z)B^{d-1},\quad B\geqslant 1 \]
 for all $Z\in \mathcal C(X\cap H_a)$. So we have
 \[\#M(X\cap H_a;B)\ll_{n}\delta B^{d-1},\quad B\geqslant 1.\]
 On the other hand, by the relation \eqref{intersected by hyperplane}, we have
 \[\#M(X;B)\leqslant\sum_{\begin{subarray}{c}a\in\mathbb Z\\|a|\leqslant B \end{subarray}}\#M(X\cap H_a;B), \quad B\geqslant1.\]
 There are at most $2B+1$ integers whose absolute values are smaller than $B$. So we have the result from the induction hypothesis.
\end{proof}
With the help of Lemma \ref{affine cone lemma}, we can prove Theorem \ref{refined Schanuel estimate for rational points}. The main idea comes from the proof of \cite[Theorem 3.1]{Browning-PM277}.
\begin{proof}[Proof of Theorem \ref{refined Schanuel estimate for rational points}]
   Let $\hat X$ be the affine cone of $X$ in $\mathbb A^{n+1}_{\mathbb Q}$. Then we have
  \[N(X;B)\leqslant \#M(\hat{X};B),\quad B\geqslant1,\]
  where the set $M(\hat X;B)$ follows the notation in the statement of Proposition \ref{affine cone lemma}.
  By \cite[Chap. I, Exercise 2.10]{GTM52}, the scheme $\hat{X}$ is of pure dimension, whose dimension is $d+1$.

   Let $\overline X$ be the projective closure of $\hat X$ in $\mathbb P^{n+1}_\Q$. Suppose $X=\proj\left(\Q[T_0,\ldots,T_n]/\a_X\right)$ and $\overline X=\proj\left(\Q[T_0,\ldots,T_{n+1}]/\a_{\overline X}\right)$. By \cite[(8.3.1.1), (8.3.1.2)]{EGAII}, we have
  \[\Q[T_0,\ldots,T_{n+1}]/\a_{\overline X}=\left(\Q[T_0,\ldots,T_n]/\a_X\right)\otimes_{\Q}\Q[T_{n+1}]\]
  as $\Q$-vector spaces.

  By \cite[Corollary 1.1.13]{Joins}, we obtain that the Hilbert function of $\overline X$ is the convolution of the Hilbert function of $X$ and that of $\proj\left(\Q[T_{n+1}]\right)$. By \cite[Lemma 1.1.12]{Joins}, we get $\deg(\hat X)=\deg(\overline X)=\deg(X)=\delta$ which follows Definition \ref{definition of degree of affine scheme}. So we get the result from Proposition \ref{affine cone lemma}.
\end{proof}

\subsection{Varieties of degree larger than $1$}\label{HB conjecture}
Let $X\hookrightarrow\mathbb P^n_K$ be an integral closed projective scheme over the number field $K$, whose dimension is $d$ and degree is $\delta$. In Theorem \ref{refined Schanuel estimate for rational points}, we give the optimal response to the case of $\delta=1$. Actually, the scheme $X$ is isomorphic to $\mathbb P^d_K$ in this case.

 We consider the cases which the degree is greater than or equal to $2$. In \cite[Conjecture 2]{Heath-Brown}, D. R. Heath-Brown conjectured that, if $d$, $\delta$ and $n$ are integers such that $d\geqslant2$, $\delta\geqslant3$, $n\geqslant4$, and $B\geqslant1$ be a real number. Then for any $\epsilon>0$, the estimate
\begin{equation}\label{strong HB conjecture}
N(X;B)\ll_{n,\epsilon,K}B^{d+\epsilon}
\end{equation}
or a weaker one
\begin{equation}\label{weak HB conjecture}
N(X;B)\ll_{n,\epsilon,K,\delta}B^{d+\epsilon}
\end{equation}
hold uniformly for every integral closed subscheme $X$ of $\mathbb P_K^n$ of degree $\delta$ and dimension $d$. He has also given a proof for the case of $\delta=2$. In order to solve this conjecture, T. Browning, D. R. Heath-Brown and P. Salberger have published several papers on this topic, see \cite{Browning_Heath05,Browning_Heath06I,Browning_Heath06II,Bro_HeathB_Salb,Salberger07} for their works on this subject and \cite[Chapter 3]{Browning-PM277} for a survey. In their former work, they imposed some technical conditions of $X$. For example, in \cite[Theorem 0.1]{Salberger07}, P. Salberger proved \cite[Conjecture 2]{Heath-Brown} when $X$ contains finitely many linear locus of codimension $1$ and $\delta\geqslant4$.

When we consider the conjecture \eqref{weak HB conjecture}, another important issue is to consider the order of $\delta$ in this estimate. This estimate will be useful for this multiplicity-counting problem, see \S\ref{counting multiplicity}.

\section{Estimate of multiplicities in a hypersurface}

In order to study the multiplicities in a projective hypersurface, first we introduce some facts about the multiplicity of a point in a hypersurface.

\subsection{Multiplicity in a section of hypersurface}
Let $k$ be an arbitrary field, and $f\in k[T_0,\ldots,T_n]$ be a non-zero homogeneous polynomial of degree $\delta$. We consider the scheme
\[X=V(f)=\proj\left(k[T_0,\ldots,T_n]/(f)\right).\]
In fact, the scheme $X$ is a pure dimensional closed subscheme of $\mathbb P^n_k$ and it is a hypersurface in it. We can prove that $X$ is of degree $\delta$ (cf. \cite[Proposition 7.6, Chap. I]{GTM52}).

Let $\alpha\in[0,\delta]\cap \mathbb N$. We denote by $\mathcal{T}^{\alpha}(f)$ be $k$-vector space spanned by all the partial derivatives of $f$ of order $\alpha$ which are of the following form
\[\dfrac{\partial^{|I|}f}{\partial T^I} = \dfrac{\partial^{i_0+\cdots+i_n}f}{\partial T_0^{i_0}\cdots\partial T_n^{i_n}}\]
for $I = (i_0,\cdots,i_n)\in\mathbb{N}^{n+1}$ with $|I| = i_0+\cdots+i_n = \alpha$. These elements are homogeneous polynomials of degree $\delta-\alpha$.

The following proposition is an explicit criterion in determining the multiplicity of a point in a hypersurface.
\begin{prop}[Corollaire 5.4, \cite{Liu-multiplicity}] \label{taylor expansion}
Let $k$ be a arbitrary field of characteristic $0$, $X\hookrightarrow\mathbb P^n_k$ be a hypersurface defined by an arbitrary non-zero homogeneous polynomial $f$ of degree $\delta$, $\eta\in X$ be an arbitrary point, and $\alpha$ be an arbitrary integer in $[0,\mu_\eta(X)-1]$. Then for every non-zero $g\in\mathcal{T}^{\alpha}(f)$, the point $\eta$ is contained in the hypersurface $X'$ defined by $g$. On the contrary, there exists a non-zero element $g'\in\mathcal{T}^{\mu_\eta(X)}(f)$, such that $\eta$ is not contained in the hypersurface defined by $g'$.
\end{prop}
In particular, if $\mu_{\eta}(X) > 1$, then for each non-zero $g \in \mathcal{T}^1(f)$, $\eta$ lies in the hypersurface defined by $g$. An immediate consequence of this proposition is that for each hypersurface $X'$ mentioned in Proposition \ref{taylor expansion}, we have $\mu_{\eta}(X') \geqslant \mu_{\eta}(X) - \alpha$, $\forall \alpha \in [0,\mu_\eta(X)-1]$.

\subsection{Construction of intersection trees}\label{construction of intersection trees}
By virtue of Proposition \ref{taylor expansion}, we can construct a family of intersection trees to solve the multiplicity-counting problem. First, we introduce the following proposition to construct the roots of these intersection trees.
\begin{prop}[Lemme 5.8, \cite{Liu-multiplicity}]\label{construction of roots}
Let $K$ be an arbitrary number field, $f$ be an arbitrary non-zero homogeneous polynomial of degree $\delta$ in $K[T_0,\ldots,T_n]$, and $s\in[0,n-2]\cap\mathbb N$. We denote by $V(f)$ the projective hypersurface defined by $f$. If the singular locus of $V(f)$ is of dimension $s$, then there exists a family of directional derivatives $g_1,\ldots, g_{n-s-1}\in\mathcal{T}^1(f)$ of $f$, such that the equality
\[\dim(V(f)\cap V(g_1)\cap\cdots\cap V(g_{n-s-1}))=s\]
is verified. In the other words, $V(f)\cap V(g_1)\cap\cdots\cap V(g_{n-s-1})$ is a complete intersection.
\end{prop}
Let $K$ and $f$ be the same as in Proposition \ref{construction of roots}. We denote by
\[X=\proj\left(K[T_0,\ldots,T_n]/(f)\right)\]
in the following argumentation. We denote by $X^{\mathrm{reg}}$ the regular locus of $X$, and by $X^{\mathrm{sing}}$ the singular locus of $X$. Following the notation and conditions in Proposition \ref{construction of roots}, we denote by $X_i$ the hypersurface $V(g_i)$ for simplicity below, where $i=1,\ldots,n-s-1$. By the Jacobian criterion (cf. \cite[Theorem 4.2.19]{LiuQing}), we have $X^\mathrm{sing}\subseteq X\cap X_1\cap\cdots\cap X_{n-s-1}$.

For every integral closed subscheme $M$ of $X$, we denote by $M^{(a)}$ the locus of the points $\xi$ in $M$ whose multiplicities $\mu_{\xi}(X)$ are equal to $\mu_M(X)$, and by $M^{(b)}$ the locus of the points $\xi$ in $M$ whose multiplicities $\mu_{\xi}(X)$ are greater than or equal to $\mu_M(X)+1$. In addition, let $L/K$ be an extension of fields, and we denote by $M^{(a)}(L)$ (\resp $M^{(b)}(L)$) the set of $L$-rational points of $M^{(a)}$ (\resp $M^{(b)}$). With this notation, we have $M(L) = M^{(a)}(L)\bigsqcup M^{(b)}(L)$.

By Proposition \ref{taylor expansion}, we obtain that $M^{(a)}$ is dense in $M$ since $M^\mathrm{reg}$ is dense in $M$ and all of them have multiplicity $\mu_M(X)$. The dimension of $M^{(b)}$ is equal to or smaller than $\dim (M)-1$.

Next, we construct a family of intersection trees $\{\mathscr T_C\}$, where $C\in\mathcal C(X\cdot X_{1,}\intersect X_{n-s-1})$. The root of the intersection tree $\mathscr T_C$ is $C$.

In order to construct those vertices whose depth are equal to or larger than $1$, let $M$ be a vertex which is already constructed in these intersection trees $\{\mathscr T_C\}$. We regard $M$ as an integral closed subscheme of $X$. Next, we consider the set $M(\overline K)$. If $M^{(b)}(\overline K)=\emptyset$, then the vertex $M$ is a leaf in one of these intersection trees.

If $M^{(b)}(\overline K) \neq \emptyset$, let $\xi \in M^{(b)}(\overline K)$, then we have $\mu_M(X) < \mu_{\xi}(X) \leqslant \delta$. By Proposition \ref{taylor expansion}, for a fixed point $\xi'\in M^{(a)}(\overline K)$, we can find some $h \in \mathcal{T}^{\delta-\mu_M(X)}(f)$, such that the hypersurface defined by $h$ does not contain $\xi'$. In this case, the hypersurface $V(h)$ does not contain the generic point of $M$. By comparing the dimensions of $V(h)$ and $M$, we obtain that $V(h)$ intersects $M$ properly. Of course we have $\deg (h)\leqslant\delta-1$. In this case, we define $V(h)$ as the label $\widetilde{M}$ of $M$. The children of $M$ hence are the irreducible components of the intersection $M\cdot\widetilde{M}$. The weights of the edges are the intersection multiplicities respectively.

For the construction that follows, all the mentioned labels are of dimension $n-1$, hence all the vertices in $\mathcal C_w$ are of dimension $s-w$, where $1\leqslant w\leqslant s$ is an integer. The construction terminates in finite steps.

The following lemma is a property of the set $\mathcal Z_*$ (see Definition \ref{def of Z_s}), which will be useful in the proof of Theorem \ref{main theorem}.

\begin{lemm}[Lemme 5.9, \cite{Liu-multiplicity}]\label{control of singular locus}
With all the notation and construction above, for every $\xi\in X^{\mathrm{sing}}(\overline{K})$, there exists at least one $Z\in\mathcal Z_*$ such that $\xi\in Z^{(a)}(\overline K)$.
\end{lemm}
\begin{rema}
In the original proof of Lemma \ref{control of singular locus} in \cite[Lemme 5.9]{Liu-multiplicity}, we work over a finite field. In fact this result remains true for the case of a number field, since the proof of \cite[Lemme 5.9]{Liu-multiplicity} only uses the assumption that the base field is perfect.
\end{rema}

\subsection{Counting Multiplicities}\label{counting multiplicity}
With the construction above, we are going to prove the following result.
\begin{theo}\label{main theorem}
  Let $K$ be an arbitrary number field, $n\geqslant2$, $\delta\geqslant1$ and $s\geqslant0$ be three integers. Then the inequality
  \begin{eqnarray*}& &\sum_{\xi\in S(X;D,B)} \mu_{\xi}(X)(\mu_{\xi}(X)-1)^{n-s-1}\\
&\leqslant&\sum_{t=0}^s\max_{Z\in\mathcal Z_t}\left\{\frac{N(Z;D,B)}{\deg(Z)}\right\}\delta(\delta-1)^{n-s+t-1},\end{eqnarray*}
is verified for all reduced hypersurfaces $X$ of $\mathbb P^n_K$ of degree $\delta$, whose dimension of singular locus is $s$. In this inequality, $S(X;D,B)$ is defined in \eqref{S(X;D,B)}, $N(X;D,B)$ is defined in \eqref{N(X;D,B)}, and $\mathcal Z_t$ is defined in Definition \ref{def of Z_s} following the construction in \S \ref{construction of intersection trees}. If $\mathcal Z_t=\emptyset$ for some $0\leqslant t\leqslant s$, we define $\max\limits_{Z\in\mathcal Z_t}\left\{\frac{N(Z;D,B)}{\deg(Z)}\right\}=0$  by convention.
\end{theo}
\begin{proof}
  Suppose that a family of intersection trees $\{\mathscr{T}_C\}$ whose roots are the elements in $\mathcal C(X\cdot X_{1}\intersect X_{n-s-1})$ has already been constructed via procedures introduced in \S \ref{construction of intersection trees}.

First, we have
\begin{equation}\label{first step}
\sum\limits_{\xi\in S(X;D,B)}\mu_\xi(X)(\mu_\xi(X)-1)^{n-s-1}=\sum\limits_{\xi\in S(X^\mathrm{sing};D,B)}\mu_\xi(X)(\mu_\xi(X)-1)^{n-s-1},
\end{equation}
since for every $\xi\in X^\mathrm{reg}$, we always have $\mu_\xi(X)=1$.

By Lemma \ref{control of singular locus}, for each $\xi\in X^{\mathrm{sing}}(\overline K)$, we can find a $Z\in\mathcal Z_*$ such that $\xi\in Z^{(a)}(\overline K)$. So we have
\begin{eqnarray}\label{final 1}
& &\sum\limits_{\xi\in S(X^\mathrm{sing};D,B)}\mu_\xi(X)(\mu_\xi(X)-1)^{n-s-1}\\
&\leqslant&\sum_{t=0}^{s}\sum_{Z\in\mathcal Z_t}\sum_{\xi\in S(Z^{(a)};D,B)}\mu_\xi(X)(\mu_\xi(X)-1)^{n-s-1}.\nonumber
\end{eqnarray}

By Proposition \ref{taylor expansion}, for every $Z\in \mathcal Z_*$, we have the inequality
\[\mu_Z(X)-1\leqslant\mu_Z(X_i)\]
for all $i=1,\ldots,n-s-1$. So we get the inequality
\begin{equation}\label{final_1'}
\mu_Z(X)(\mu_Z(X)-1)^{n-s-1}\leqslant\mu_Z(X)\mu_Z(X_{1})\cdots\mu_Z(X_{n-s-1}).
\end{equation}

 By Proposition \ref{grassmanne} and the inequality \eqref{final_1'}, we have
\begin{eqnarray}\label{final 2}
& &\sum_{Z\in\mathcal Z_t}\mu_Z(X)(\mu_Z(X)-1)^{n-s-1}\deg(Z) \\
& \leqslant & \sum_{Z\in\mathcal Z_t}\mu_Z(X)\mu_Z(X_{1})\cdots\mu_Z( X_{n-s-1})\deg(Z) \nonumber \\
& \leqslant & \deg(X)\prod_{i=1}^{n-s-1}\deg(X_i)\prod_{j=0}^{t-1}\max_{\widetilde{Z}\in \mathcal C'_t}\{\deg(\widetilde{Z})\} \leqslant\delta(\delta-1)^{n-s+t-1}, \nonumber
\end{eqnarray}
for each $t=0,\ldots,s$, since all the labels in $\mathcal C'_*$ are of degree equal to or smaller than $\delta-1$.

Combine the inequalities \eqref{final 1} and \eqref{final 2}, we obtain that
\begin{eqnarray}\label{final step}
  & &\sum_{t=0}^{s}\sum_{Z\in\mathcal Z_t}\sum_{\xi\in S(Z^{(a)};D,B)}\mu_\xi(X)(\mu_\xi(X)-1)^{n-s-1}\\
  &=&\sum_{t=0}^{s}\sum_{Z\in\mathcal Z_t}\mu_Z(X)(\mu_Z(X)-1)^{n-s-1}N(Z^{(a)};D,B)\nonumber\\
  &\leqslant&\sum_{t=0}^{s}\sum_{Z\in\mathcal Z_t}\mu_Z(X)(\mu_Z(X)-1)^{n-s-1}N(Z;D,B)\nonumber\\
  &\leqslant&\sum_{t=0}^{s}\max_{Z\in\mathcal Z_t}\left\{\frac{N(Z;D,B)}{\deg(Z)}\right\}\left(\sum_{Z\in\mathcal Z_t}\mu_Z(X)(\mu_Z(X)-1)^{n-s-1}\deg(Z)\right)\nonumber\\
   &\leqslant&\sum_{t=0}^s\max_{Z\in\mathcal Z_t}\left\{\frac{N(Z;D,B)}{\deg(Z)}\right\}\delta(\delta-1)^{n-s+t-1}.\nonumber
\end{eqnarray}
By the inequalities \eqref{first step}, \eqref{final 1} and \eqref{final step}, we prove the result.
\end{proof}
\begin{rema}\label{choice of counting function}
  We keep all the notation and conditions in Theorem \ref{main theorem}. Let $f:\mathbb N^+\rightarrow \mathbb N$ be an increasing function which is asymptotic to a polynomial whose degree is smaller than $n-s-1$, and it satisfies $f(1)=0$. Then there exists a constant $C_f$ depending only on the function $f$, such that $f(x)\leqslant C_f\cdot x(x-1)^{n-s-1}$ for all $x\geqslant1$. Then the inequality
  \begin{eqnarray*}\sum_{\xi\in S(X;D,B)}f(\mu_\xi(X))&\leqslant& C_f\sum_{\xi\in S(X;D,B)}\mu_\xi(X)(\mu_\xi(X)-1)^{n-s-1}
  \end{eqnarray*}
  is verified for all funtions $f$ satisfying the above conditions, where $D\in\mathbb N^+$ and $B\geqslant1$. If we do not care about the constant depending on the above function $f(x)$, for this kind of counting multiplicities problem, it is enough to consider the counting function $f(x)=x(x-1)^{n-s-1}$ only, which is considered in Theorem \ref{main theorem}.

  If we consider another increasing counting function $g:\mathbb N^+\rightarrow\mathbb N$ which is asymptotic to a polynomial whose degree is smaller than $n-s-1$, and we do not suppose the condition $g(1)=0$ any longer. Then we have
  \begin{eqnarray*}
    \sum_{\xi\in S(X;D,B)}g(\mu_\xi(X))&=&g(1)N(X^{\mathrm{reg}};D,B)+\sum_{\xi\in S(X;D,B)}\left(g(\mu_\xi(X))-g(1)\right)\\
    &\leqslant&g(1)N(X;D,B)+\sum_{\xi\in S(X;D,B)}\left(g(\mu_\xi(X))-g(1)\right).
  \end{eqnarray*}
  We consider the sum
  \[\sum\limits_{\xi\in S(X;D,B)}\left(g(\mu_\xi(X))-g(1)\right)\]
   by the above discussion, and consider the term $g(1)N(X;D,B)$ as the classical problem of counting algebraic points or rational points (when $D=1$). By the fact that $X^\mathrm{reg}$ and $X$ are birational equivalent, this estimate is appropriate.
\end{rema}
\subsection{The case of rational points}
For the sum considered in Theorem \ref{main theorem}, now we consider the case of counting multiplicities of rational points.  If we want a uniform upper bound of it, we have the following result, which is a corollary of Theorem \ref{main theorem} combined with the generalized Schanuel's estimate (Theorem \ref{refined Schanuel estimate for rational points}).
\begin{coro}\label{main corollary}
Let $n\geqslant2$, $\delta\geqslant1$ and $s\geqslant0$ be three integers. Then the estimate
\[\sum_{\xi\in S(X;B)}\mu_\xi(X)(\mu_\xi(X)-1)^{n-s-1}\ll_{n}\delta^{n-s}\max\{\delta-1,B\}^{s+1},\quad B\geqslant 1\]
holds uniformly for reduced hypersurfaces $X$ of $\mathbb P^n_{\Q}$ of degree $\delta$ whose singular locus is of dimension $s$, where $S(X;B)$ is defined in $\eqref{S(X;B)}$.
\end{coro}
\begin{proof}
By the argumentations in Theorem \ref{refined Schanuel estimate for rational points}, we obtain the estimate
\[\frac{N(Z;B)}{\deg(Z)}\ll_{n}B^{\dim(Z)+1},\quad B\geqslant 1\]
holds uniformly for $Z\in\mathcal Z_t$, where $t=0,\ldots,s$ following the construction in \S\ref{construction of intersection trees}. Combine the above inequality with the estimate in Theorem \ref{main theorem} and the fact that $\dim(Z)<n$ and $s<n$, we obtain the result.
\end{proof}

\begin{exem}\label{cylinder}
Let $X'\hookrightarrow\mathbb P^2_\Q$ be a reduced plane curve of degree $\delta$, which is defined by the homogeneous equation $f(T_0,T_1,T_2)=0$. Suppose that $X'$ has a $\Q$-rational point of multiplicity $\delta$. We consider $f$ to be a homogeneous polynomial of degree $\delta$ in $\Q[T_0,\ldots,T_n]$ for an integer $n\geqslant3$. Then $f$ defines a hypersurface in $\mathbb P^n_\Q$, denoted by $X$. Without loss of generality, we suppose that $[1:0:0]$ is the projective coordinate of this singular point of $X'$. Then we have
\begin{eqnarray*}
X^{\mathrm{sing}}(\Q)&=&\{[x_0:\cdots:x_n]\in\mathbb P^n_{\Q}(\Q)|\;x_0=1,x_1=0,x_2=0\}\cup\\
& &\{[x_0:\cdots:x_n]\in\mathbb P^n_{\Q}(\Q)|\;x_0=x_1=x_2=0\},
\end{eqnarray*}
where all singular points of $X$ are of multiplicity $\delta$, and $X^\mathrm{sing}$ is considered to be a reduced closed subscheme of $\mathbb P^n_\Q$. By the equality \eqref{Schanuel over Q}, we have
\begin{equation*}
N(X^\mathrm{sing};B)=N(\mathbb P^{n-2}_\Q;B)=\frac{2^{n-2}}{\zeta(n-1)}B^{n-1}+o(B^{n-1}),\quad B\rightarrow+\infty.
\end{equation*}
Then we have the following asymptotic estimate
\begin{equation*}
\sum\limits_{\xi\in S(X;B)}\mu_\xi(X)(\mu_\xi(X)-1)=\delta(\delta-1)N(X^\mathrm{sing};B)=O_{n}\left(\delta^2B^{n-1}\right)
\end{equation*}
for each hypersurface $X$ satisfying the above conditions.

From this example, the order of $\delta$ in Corollary \ref{main corollary} is optimal when $\dim(X^\mathrm{sing})=n-2$ and $n\geqslant3$. More generally, if $X^\mathrm{sing}$ contains a linear locus of multiplicity $\delta$ in $X$, we can get the maximal order of $\delta$ and $\max\{B,\delta-1\}$ in the estimate of Corollary \ref{main corollary} when $B\geqslant\delta-1$.
\end{exem}

Let $K$ be a number field, and $X\hookrightarrow\mathbb P^n_K$ be a fixed hypersurface of degree $\delta$. To attack this kind of counting multiplicities problem by applying Theorem \ref{main theorem}, the key point is the uniform estimate of the term
\[\max_{Z\in\mathcal Z_t}\left\{\frac{N(Z;D,B)}{\deg(Z)}\right\}\]
for all possible $Z\in\mathcal Z_t$ ($t=1,\ldots,s$) in the intersection trees constructed above. In Theorem \ref{refined Schanuel estimate for rational points}, we give a description of this term for the case of $K=\Q$, and we obtain Corollary \ref{main corollary} through it. If all the irreducible components of $X^{\mathrm{sing}}$ are of degree strictly greater than $1$, we can use the estimates introduced in \S\ref{HB conjecture} to get a better estimate of
\[\sum_{\xi\in S(X;B)}\mu_\xi(X)(\mu_\xi(X)-1)^{n-s-1}\]
than that given in Corollary \ref{main corollary}. The reason is that we have a better estimate of
 \[\max\limits_{Z\in\mathcal Z_t}\left\{\frac{N(Z;B)}{\deg(Z)}\right\}\]
  than that given in Theorem \ref{refined Schanuel estimate for rational points} in this case, where $Z\in\mathcal Z_t$ is defined same as in Theorem \ref{main theorem}.
\begin{exem}\label{example using HB conj}
  Let $\delta\geqslant3$ and $n\geqslant2$ be two integers. Let
  \begin{equation}\label{hypersurface Z}
    Z\hookrightarrow\mathbb P^{n+2}_\Q=\proj\left(\Q[X,Y,T_0,\ldots,T_n]\right)
  \end{equation}
   be the hypersurface defined by the homogeneous polynomial
   \begin{equation}\label{hypersurface Z2}
     F(X,Y,T_0,\ldots,T_n)=Y^\delta+Xf(T_0,T_1,\cdots,T_n),
   \end{equation}
   where $f(T_0,\ldots,T_n)$ is an irreducible homogeneous polynomial of degree $\delta-1$ which defines a smooth hypersurface in $\mathbb P^n_K$, noted by $Z'$ this hypersurface. The polynomial $F(X,Y,T_0,\ldots,T_n)$ is irreducible. By \cite[Exercise 2.4.1]{LiuQing}, we obtain that the hypersurface $Z$ is integral. By \cite[Theorem 1, Corollary]{Browning_Heath06II}, for all integers $\delta\geqslant3$ and $n\geqslant2$, and for any $\epsilon>0$, the estimate
  \[N(Z';B)\ll_{n,\delta,\epsilon}B^{n-1+\epsilon},\quad B\geqslant1\]
holds uniformly for every smooth hypersurface $Z'$.

  Meanwhile, by Jacobian criterion (cf. \cite[Theorem 4.2.19]{LiuQing}), the singular locus of $X$ is defined by
   \[0=F(X,Y,T_0,\ldots,T_n)=\delta Y^{\delta-1}=f(T_0,\ldots,T_n)=X\frac{\partial f}{\partial T_0}=\cdots=X\frac{\partial f}{\partial T_n}.\]
   Because the hypersurface $Z'$ is smooth over $\Q$, the polynomials $f,\frac{\partial f}{\partial T_0},\ldots,\frac{\partial f}{\partial T_n}$ have no common non-zero solutions. Hence for each $\xi\in Z^{\mathrm{sing}}(\overline \Q)$, the projective coordinate $[x:y:t_0:\cdots:t_n]$ of $\xi$ satisfies $x=y=0$ and $[t_0:\cdots:t_n]\in Z'(\overline \Q)$, and every singular points is of multiplicity $2$ in $Z$. By definition, $Z^\mathrm{sing}$ is of codimension $2$ in $Z$, whose dimension is $n-1$. We consider the sum in Theorem \ref{main theorem} for this example, for arbitrary integers $\delta\geqslant3$, $n\geqslant2$, and for all $\epsilon>0$, the equality
   $$\sum_{\xi\in S(Z;B)}\mu_\xi(X)(\mu_\xi(X)-1)^2=2N(Z';B),\quad B\geqslant1$$
is verified for all hypersurfaces defined by the method in \eqref{hypersurface Z}, \eqref{hypersurface Z2}.

We follow the construction in \S\ref{construction of intersection trees}, where the proper intersection of $Z$, $V\left(\frac{\partial F}{\partial X}\right)$ and $V\left(\frac{\partial F}{\partial Y}\right)$ generates the only root of the intersection tree, and it has no descendent. Then we apply Theorem \ref{main theorem} to this case directly, and we obtain the inequality
\[\sum_{\xi\in S(Z;B)}\mu_\xi(X)(\mu_\xi(X)-1)^2\leqslant\delta(\delta-1)^2\frac{N(Z';B)}{\delta-1}=\delta(\delta-1)N(Z';B), \quad B\geqslant1.\]
This is an example which satisfies the upper bound given in Theorem \ref{main theorem}, since $\delta\geqslant3$.

By \cite[Theorem 1, Corollary]{Browning_Heath06II}, the estimate
\[\sum_{\xi\in S(Z;B)}\mu_\xi(X)(\mu_\xi(X)-1)^2=2N(Z';B)\ll_{n,\delta,\epsilon}B^{n-1+\epsilon}, \quad B\geqslant1\]
holds uniformly for all hypersurfaces defined by the method in \eqref{hypersurface Z}, \eqref{hypersurface Z2} and all $\epsilon>0$. In this case, the above estimate gives a better dependance on $B$ than that given in Corollary \ref{main corollary}. But in this estimate, we have no description of the order of $\delta$, since we cannot control the order of $\delta$ in the above estimate to the extent of our current knowledge.
\end{exem}

Similar to \cite[Conjecture 5.13]{Liu-multiplicity}, we propose the following conjecture.
\begin{conj}
Let $K$ be a number field, and $\delta\geqslant1$, $d\geqslant1$, $s\geqslant0$ be three integers. The estimate
\[\sum_{\xi\in S(X;B)}\mu_\xi(X)(\mu_\xi(X)-1)^{d-s}\ll_{n,K}\delta^{d-s+1}B^{s+1}\]
holds uniformly for all reduced pure dimensional closed subschemes $X$ of $\mathbb P^n_K$ of dimension $d$ and degree $\delta$, whose dimension of singular locus is $s$, where $S(X;B)$ is defined in \eqref{S(X;B)}.
\end{conj}


\backmatter

\bibliography{wen}

\def\cprime{$'$} \def\cprime{$'$}
\providecommand{\bysame}{\leavevmode ---\ }
\providecommand{\og}{``}
\providecommand{\fg}{''}
\providecommand{\smfandname}{\&}
\providecommand{\smfedsname}{\'eds.}
\providecommand{\smfedname}{\'ed.}
\providecommand{\smfmastersthesisname}{M\'emoire}
\providecommand{\smfphdthesisname}{Th\`ese}
\begin{thebibliography}{10}

\bibitem{Browning_Heath05}
{\scshape T.~D. Browning {\normalfont \smfandname} D.~R. Heath-Brown} -- {\og
  Counting rational points on hypersurfaces\fg}, \emph{Journal f\"ur die Reine
  und Angewandte Mathematik} \textbf{584} (2005), p.~83--115.

\bibitem{Browning_Heath06I}
\bysame , {\og The density of rational points on non-singular hypersurfaces.
  {I}\fg}, \emph{The Bulletin of the London Mathematical Society} \textbf{38}
  (2006), no.~3, p.~401--410.

\bibitem{Browning_Heath06II}
\bysame , {\og The density of rational points on non-singular hypersurfaces.
  {II}\fg}, \emph{Proceedings of the London Mathematical Society. Third Series}
  \textbf{93} (2006), no.~2, p.~273--303, With an appendix by J. M. Starr.

\bibitem{Bro_HeathB_Salb}
{\scshape T.~D. Browning, D.~R. Heath-Brown {\normalfont \smfandname}
  P.~Salberger} -- {\og Counting rational points on algebraic varieties\fg},
  \emph{Duke Mathematical Journal} \textbf{132} (2006), no.~3, p.~545--578.

\bibitem{Browning-PM277}
{\scshape T.~D. Browning} -- \emph{Quantitative arithmetic of projective
  varieties}, Progress in Mathematics, vol. 277, Birkh\"auser Verlag, Basel,
  2009.

\bibitem{Joins}
{\scshape H.~Flenner, L.~O'Carroll {\normalfont \smfandname} W.~Vogel} --
  \emph{Joins and intersections}, Springer Monographs in Mathematics,
  Springer-Verlag, Berlin, 1999.

\bibitem{Fulton2}
{\scshape W.~Fulton} -- \emph{Algebraic curves. {A}n introduction to algebraic
  geometry}, W. A. Benjamin, Inc., New York-Amsterdam, 1969, Notes written with
  the collaboration of Richard Weiss, Mathematics Lecture Notes Series.

\bibitem{Fulton}
\bysame , \emph{Intersection theory}, second \smfedname, Ergebnisse der
  Mathematik und ihrer Grenzgebiete. 3. Folge. A Series of Modern Surveys in
  Mathematics [Results in Mathematics and Related Areas. 3rd Series. A Series
  of Modern Surveys in Mathematics], vol.~2, Springer-Verlag, Berlin, 1998.

\bibitem{GaoXiaThesis}
{\scshape X.~Gao} -- \emph{On {N}orthcott's theorem}, ProQuest LLC, Ann Arbor,
  MI, 1995, Thesis (Ph.D.)--University of Colorado at Boulder.

\bibitem{EGAII}
{\scshape A.~Grothendieck} -- {\og \'{E}l\'ements de g\'eom\'etrie
  alg\'ebrique. {II}. \'{E}tude globale \'el\'ementaire de quelques classes de
  morphismes\fg}, \emph{Institut des Hautes \'Etudes Scientifiques.
  Publications Math\'ematiques} (1961), no.~8, p.~222.

\bibitem{Guignard2017}
{\scshape Q.~Guignard} -- {\og Counting algebraic points of bounded height on
  projective spaces\fg}, \emph{Journal of Number Theory} \textbf{170} (2017),
  p.~103--141.

\bibitem{GTM52}
{\scshape R.~Hartshorne} -- \emph{Algebraic geometry}, Springer-Verlag, New
  York, 1977, Graduate Texts in Mathematics, No. 52.

\bibitem{Heath-Brown}
{\scshape D.~R. Heath-Brown} -- {\og The density of rational points on curves
  and surfaces\fg}, \emph{Annals of Mathematics. Second Series} \textbf{155}
  (2002), no.~2, p.~553--595.

\bibitem{Hindry}
{\scshape M.~Hindry {\normalfont \smfandname} J.~H. Silverman} --
  \emph{Diophantine geometry, an introduction}, Graduate Texts in Mathematics,
  vol. 201, Springer-Verlag, New York, 2000.

\bibitem{Laumon1975}
{\scshape G.~Laumon} -- {\og Degr\'e de la vari\'et\'e duale d'une hypersurface
  \`a singularit\'es isol\'ees\fg}, \emph{Bulletin de la Soci\'et\'e
  Math\'ematique de France} \textbf{104} (1976), no.~1, p.~51--63.

\bibitem{LeRudulier2014}
{\scshape C.~Le~Rudulier} -- {\og Points alg\'ebriques de hauteur born\'ee sur
  la droite projective\fg}, \emph{J. Th\'eor. Nombres Bordeaux} \textbf{26}
  (2014), no.~3, p.~789--813.

\bibitem{Liu-multiplicity}
{\scshape C.~Liu} -- {\og Comptage des multiplicit\'es dans une hypersurface
  sur un corps fini\fg}, \emph{\url{arxiv:1606.09337}} (2016).

\bibitem{LiuQing}
{\scshape Q.~Liu} -- \emph{Algebraic geometry and arithmetic curves}, Oxford
  Graduate Texts in Mathematics, vol.~6, Oxford University Press, Oxford, 2002,
  Translated from the French by Reinie Ern{\'e}, Oxford Science Publications.

\bibitem{MasserVaaler2007}
{\scshape D.~Masser {\normalfont \smfandname} J.~D. Vaaler} -- {\og Counting
  algebraic numbers with large height. {II}\fg}, \emph{Transactions of the
  American Mathematical Society} \textbf{359} (2007), no.~1, p.~427--445.

\bibitem{Neukirch}
{\scshape J.~Neukirch} -- \emph{Algebraic number theory}, Grundlehren der
  Mathematischen Wissenschaften [Fundamental Principles of Mathematical
  Sciences], vol. 322, Springer-Verlag, Berlin, 1999, Translated from the 1992
  German original and with a note by Norbert Schappacher, With a foreword by G.
  Harder.

\bibitem{Salberger07}
{\scshape P.~Salberger} -- {\og On the density of rational and integral points
  on algebraic varieties\fg}, \emph{Journal f\"ur die Reine und Angewandte
  Mathematik} \textbf{606} (2007), p.~123--147.

\bibitem{Schanuel}
{\scshape S.~H. Schanuel} -- {\og Heights in number fields\fg}, \emph{Bulletin
  de la Soci\'et\'e Math\'ematique de France} \textbf{107} (1979), no.~4,
  p.~433--449.

\bibitem{Schmidt1995}
{\scshape W.~M. Schmidt} -- {\og Northcott's theorem on heights. {II}. {T}he
  quadratic case\fg}, \emph{Acta Arithmetica} \textbf{70} (1995), no.~4,
  p.~343--375.

\bibitem{SerreLocAlg}
{\scshape J.-P. Serre} -- \emph{Local algebra}, Springer Monographs in
  Mathematics, Springer-Verlag, Berlin, 2000, Translated from the French by
  CheeWhye Chin and revised by the author.

\end{thebibliography}
\bibliographystyle{smfplain}

\end{document}